\documentclass[11pt]{amsart}
\usepackage{amscd,amssymb,palatino}
\usepackage[utf8]{inputenc}
\usepackage[english]{babel}
\usepackage{caption}
\usepackage{etex}
\usepackage[leqno]{amsmath}%puts equation number on the left
\usepackage{amssymb}
\usepackage{color}
\usepackage{shadow}
\usepackage{epsfig}
\usepackage{epic}
\usepackage{eepic}
\usepackage{graphics}
\usepackage{graphicx}
\usepackage{calc}
\usepackage{url}
\usepackage{enumitem}

\usepackage{tikz}
\usetikzlibrary{arrows,decorations,patterns,positioning,automata,shadows,fit,shapes,calc,decorations.markings,backgrounds,scopes,decorations.text}
\usepackage{tipa}

\DeclareMathOperator{\rank}{rank}

\DeclareMathOperator{\curl}{curl}

\DeclareMathOperator{\Hess}{Hess}
\usepackage{calc}
\usepackage[utf8]{inputenc}

\newcommand\1{\lower 9pt\hbox{\underbar{}}}
\numberwithin{equation}{section}
\newtheorem*{theorem*}{Main Theorem}
\newtheorem {theorem}[equation]                   {Theorem}
\newtheorem {lemma}[equation]           {Lemma}
\newtheorem {proposition}[equation]     {Proposition}
\newtheorem {conjecture}[equation]      {Conjecture}

\newtheorem {corollary}[equation]       {Corollary}

\theoremstyle{definition}
\newtheorem {definition}[equation]{Definition}
\newtheorem {remark}[equation]          {Remark}

\newtheorem {example}[equation]         {Example}

\def\R{\mathbb{R}}

\newcommand{\norm}[1]{\left\lVert#1\right\rVert}
\def\curl{\operatorname{curl}}

\begin{document}
 %\title{A note on the singular Weinstein conjecture and $b$-Beltrami vector fields}
\title{On the singular Weinstein conjecture and the existence of escape orbits for $b$-Beltrami fields}

 \author{Eva Miranda}\address{ Eva Miranda,
Laboratory of Geometry and Dynamical Systems, Departament of Mathematics, EPSEB, Universitat Polit\`{e}cnica de Catalunya BGSMath Barcelona Graduate School of
Mathematics in Barcelona and
\\ IMCCE, CNRS-UMR8028, Observatoire de Paris, PSL University, Sorbonne
Universit\'{e}, 77 Avenue Denfert-Rochereau,
75014 Paris, France }
  \email{ eva.miranda@upc.edu}
\thanks{E. M. is supported by the Catalan Institution for Research and Advanced Studies via an ICREA Academia Prize 2016. Eva Miranda and Cédric Oms are supported by the grants reference number MTM2015-69135-P (MINECO/FEDER) and reference number 2017SGR932 (AGAUR) and the project PID2019-103849GB-I00 / AEI / 10.13039/501100011033.}
\author{Cédric Oms}
\address{ C\'edric Oms,
Laboratory of Geometry and Dynamical Systems, Department of Mathematics, EPSEB, Universitat Polit\`{e}cnica de Catalunya BGSMath Barcelona Graduate School of
Mathematics in Barcelona}
 \email{ cedric.oms@upc.edu}
\thanks{C. O. has been supported by an FNR-AFR PhD predoctoral grant (project GLADYSS) until October 2nd, 2020 and by a SECTI-Postdoctoral grant financed by Eva Miranda's ICREA Academia immediately after.}

\author{Daniel Peralta-Salas} \address{Daniel Peralta-Salas, Instituto de Ciencias Matem\'aticas, Consejo Superior de Investigaciones Cient\'ificas,
28049 Madrid, Spain}
 \email{dperalta@icmat.es}
\thanks{D. P.-S. is supported by the grants MTM PID2019-106715GB-C21 (MICINN) and Europa Excelencia EUR2019-103821 (MCIU). This work is supported in part by the
ICMAT--Severo Ochoa grant SEV-2015-0554 and the CSIC grant 20205CEX001.}

\begin{abstract}
Motivated by Poincaré's orbits going to infinity in the (restricted) three-body problem (see \cite{poincare} and \cite{chenciner}), we investigate the generic existence of heteroclinic-like orbits in a neighbourhood of the critical set of a $b$-contact form. This is done by using the singular counterpart~\cite{CMP} of Etnyre--Ghrist's contact/Beltrami correspondence~\cite{EG}, and genericity results concerning eigenfunctions of the Laplacian established by Uhlenbeck \cite{uhlenbeck}. Specifically, we analyze the $b$-Beltrami vector fields on $b$-manifolds of dimension $3$ and prove that for a generic asymptotically exact $b$-metric they exhibit escape orbits. We also show that a generic asymptotically symmetric $b$-Beltrami vector field on an asymptotically flat $b$-manifold has a generalized singular periodic orbit and at least $4$ escape orbits. Generalized singular periodic orbits are trajectories of the vector field whose $\alpha$- and $\omega$-limit sets intersect the critical surface.  These results are a first step towards proving the singular Weinstein conjecture.
\end{abstract}
	
\maketitle
	
\section{Introduction}\label{S:intro}

{Chazy established in 1922 \cite{chazyens} that the solutions of the three body problem as time tends to infinity can be of four different types: hyperbolic, parabolic, bounded and oscillatory. Oscillatory motions happen when the body with negligible mass escapes from any bounded region and returns infinitely often to it. Hyperbolic motions can often be compactified as heteroclinic-like orbits. The compactification can be understood as a regularization transformation that includes singularities in the symplectic structure that models this problem.}
{In this regularization, oscillatory motions have sequences of points tending to an invariant hypersurface corresponding to the singular set of those geometric structures.}

{When it comes to Hamiltonian dynamics it is often convenient to restrict the dynamics to a level-set of the Hamiltonian. Whenever the Liouville vector field is transverse to this level set, it induces a contact structure on it. This idea has been applied successfully to several problems in celestial mechanics such as the restricted three-body problem. For instance in \cite{AFKP} the authors prove the existence of periodic orbits on any level set $H=c$ with $c< H(L_1)$, where $L_1$ is the first Lagrange point as an application of the Weinstein conjecture to the induced contact manifold.
  The positive answer to Weinstein conjecture also yields the existence of periodic orbits in the context of hydrodynamics.}

	Indeed, as suggested by Sullivan and developed by Etnyre and Ghrist~\cite{EG}, there is a one-to-one correspondence between Reeb fields and non-vanishing Beltrami vector fields (a particular type of stationary fluid flows). Accordingly, the positive answer to the Weinstein conjecture implies  the existence of periodic orbits for any non-vanishing Beltrami vector field on a compact $3$-manifold. Vice versa, analytic techniques from hydrodynamics can provide a deep insight on the dynamics and topology of contact manifolds, as is shown for instance in~\cite{peraltaradu}. By virtue of this correspondence, both fields can potentially benefit from mutual development. This connection gained a new impulse recently with~\cite{CMPP} where universality features of Euler flows were studied using the h-principle in contact geometry. In the context of stationary fluid flows in even dimensions, Ginzburg and Khesin established a beautiful connection with symplectic geometry and integrable systems~\cite{GK2,GK1}.

The aforementioned Beltrami/contact picture has recently been generalized by including codimension one singularities. These generalized contact structures are known as \emph{$b$-contact structures} and the codimension one hypersurface, which is invariant under the flow of the Reeb field, is called \emph{critical hypersurface}. The geometry of this generalization of contact structures has been extensively studied in~\cite{MO1}, and the relation of those to Beltrami dynamics was initiated in~\cite{CMP}, proving a one-to-one correspondence between Reeb fields of $b$-contact forms and non-vanishing Beltrami vector fields on $b$-manifolds (\emph{$b$-Beltrami vector fields}). In view of this correspondence, it is important to understand the dynamical behavior of Reeb fields on $b$-contact manifolds. Those manifolds can be viewed as open regular contact manifolds satisfying certain behavior at the open ends. Due to the lack of compactness, the dynamics is fundamentally different from the smooth case. As it is shown in~\cite{MO2}, there are compact $b$-contact manifolds in any dimension, without periodic orbits away from the critical hypersurface, which contrasts the Weinstein conjecture.
Although this non-existence of periodic Reeb orbits away from $Z$ can be thought of as a counterexample to the Weinstein conjecture, a careful analysis of the intriguing dynamics on those manifolds lead the authors in~\cite{MO2} to conjecture that the next best case scenario holds:

\begin{conjecture}[Singular Weinstein conjecture, \cite{MO2}]\label{conj}
Let $(M,\alpha)$ be a compact $b$-contact manifold with critical hypersurface $Z$. Then the Reeb field $R_\alpha$ has a singular periodic orbit. More precisely, there is a Reeb orbit $\gamma:\R\to M \setminus Z$ such that $\lim_{t\to \pm \infty} \gamma(t)=p_{\pm}\in Z$ and $R_\alpha({p_\pm})=0$.
\end{conjecture}

A motivating example for this conjecture is the following, which is the $b$-analogue of the well known $ABC$ Beltrami fields~\cite{AK} on the flat torus $\mathbb T^3$:

%\begin{example}\label{ex:nopoT3}
%Consider the $3$-torus with the $b$-contact form given by $\alpha=\cos y  dx+\sin y \frac{dz}{\sin z}$. The Reeb vector field associated to this $b$-contact form is given by $R_\alpha=\cos y\partial_x+\sin y \sin z \partial_z $.  As proved in \cite{MO2}, there are infinitely many singular periodic Reeb orbits.
%\end{example}

\begin{example}\label{ex:T3flat}
Consider on the $3$-torus $\mathbb{T}^3$ the $b$-vector field given by
$$X=C\cos y \partial_x+B\sin x \partial_y+(C\sin y+B\cos x)\sin z \partial_z$$
where $|B|\neq |C|$ are two constants. This is a $b$-Beltrami vector field with constant proportionality factor $1$ for the globally flat $b$-metric given by $g=dx^2+dy^2+\frac{dz^2}{\sin^2 z}$, on the $b$-manifold $(\mathbb T^3,Z)$, where the (disconnected) critical surface is $Z=\{z=0\}\cup \{z=\pi\}$. Accordingly, $X$ can be viewed as the Reeb vector field (up to reparametrization) associated to the $b$-contact form $\alpha=g(X,\cdot)$. We claim that for all values of $B$ and $C$, $|B|\neq |C|$, there exist singular periodic orbits. The restriction of $X$ on $Z$ is a Hamiltonian vector field, with Hamiltonian function given by $H=-X_z=-C\sin y-B\cos x$, which is a first integral of $X$. Hence the integral curve of $X$ through the point $(x_0,y_0,z_0)$ satisfies $\dot{z}(t)=H(x,y)\sin z(t)=H(x_0,y_0) \sin z(t)$, so the explicit expression of $z(t)$ is given by $z(t)=2\cot^{-1}(\exp(c-H(x_0,y_0)t))$, where $c$ is a constant such that $z(0)=z_0$. Let us now analyze the singular periodic orbits of $X$. On each of the two connected components of the critical surface, there are four critical points of $H$, given in $(x,y)\in \mathbb T^2$ coordinates by $p_1=(0,\frac{\pi}{2})$, $p_2=(\pi,\frac{\pi}{2})$, $p_3=(0,\frac{3\pi}{2})$ and $p_4=(\pi,\frac{3\pi}{2})$. Since $|B|\neq |C|$, it is easy to check that $H(p_k)\neq 0$. The integral curve with initial condition $(p_k,z_0)$, where $z_0\in (0,\pi)$, is then given by $\gamma(t)=(p_k,z(t))$. Assume that $H(p_k)>0$ (the opposite case is similar). Then $\lim_{t\to \infty} z(t)=0$ and $\lim_{t\to -\infty} z(t)=\pi$, thus implying that $\gamma$ is a singular periodic orbit. As the case when $z_0\in (\pi, 2\pi)$ is analogous, we obtain $8$ singular periodic orbits; in particular, the singular Weinstein conjecture is satisfied for all $|B|\neq |C|$.
\end{example}

In this article we will study the existence of singular periodic orbits (and generalizations) on $b$-contact manifolds, via the analysis of the integral curves of $b$-Beltrami vector fields. {This is, in particular, the first instance where the (singular) contact/Beltrami correspondence is used in the reverse direction.} Since $b$-Beltrami fields are stationary solutions of the Euler equations on manifolds with cylindrical ends (the $b$-model), our results are also of interest in the study of incompressible fluid flows in equilibrium.

\emph{Singular periodic orbits} are a particular case of \emph{escape orbits}, which are Reeb orbits $\gamma\subset M\setminus Z$ that tend in forward (or backward) time to an equilibrium point, that is $\gamma \subset M \setminus Z$ such that $\lim_{t\to \infty}\gamma(t)=p$ where $p$ is a zero of $R_\alpha$ in $Z$ (respectively $\lim_{t\to -\infty}\gamma(t)=p$). From a dynamical point of view, the singular periodic orbits are contained in the intersection of the unstable manifold of $p_-$ and the stable manifold of $p_+$. Since the existence of escape orbits is a necessary condition for the existence of singular periodic orbits, establishing the existence of escape orbits is a first step towards proving the singular Weinstein conjecture.

\begin{figure}[hbt!]
\begin{center}
\begin{tikzpicture}[scale=2]
\fill[gray!20] (-1.5,-0.3) rectangle +(3,0.6);
 \draw (1,0) arc (0:180:1 and 1);
% \draw (-1,0) -- (1,0);
 \draw[dashed,color=red] (-1.5,0) --  (1.5,0);
 \fill (1,0) circle[radius=0.5pt];
 \fill (-1,0) circle[radius=0.5pt];
 (-1,-0.7) -- (-1,0.7) -- cycle;
\fill [shift={(-0.05,1)},scale=0.05,rotate=180]   (0,0) -- (-1,-0.7) -- (-1,0.7) -- cycle; %arrow up
 \draw[red] (1.6,0) node {$Z$};
 \draw (0.2,1.2) node {$\gamma$};
 \draw[thick,green] (1,0) arc (0:18:1 and 1);
 \draw[thick,green] (-1,0) arc (0:-18:-1 and -1);
\end{tikzpicture}
\caption{Singular periodic orbit vs. Escape orbits (in green)}
\label{fig:singularperiodic}
\end{center}
\end{figure}
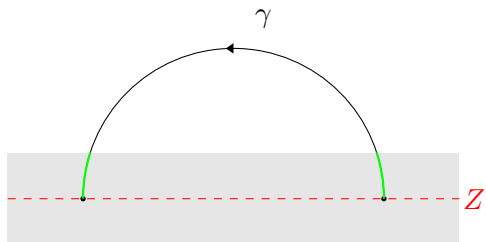

%The goal of the present paper is to address Conjecture~\ref{conj} in the three-dimensional case using the aforementioned Beltrami/contact correspondence. Specifically, let $(M,Z)$ be a compact $b$-manifold of dimension $3$ endowed with a $b$-metric $g$. This defines a Riemannian $b$-manifold $(M,Z,g)$. The critical surface $Z$ may consist of several connected components. Our first main result shows that a $b$-Beltrami vector field exhibits escape orbits on a generic asymptotically exact $b$-manifold (see Definition~\ref{def:bexactmetric}). More precisely, let $\mathcal G^k_b$ be the space of $C^k$, $k\geq 2$, asymptotically exact $b$-metrics on $(M,Z)$, endowed with the $C^k$-topology of metrics. Then:

The goal of the present paper is to address Conjecture~\ref{conj} in the three-dimensional case using the aforementioned Beltrami/contact correspondence. Specifically, let $(M,Z)$ be a compact $b$-manifold of dimension $3$ endowed with a $b$-metric $g$. This defines a Riemannian $b$-manifold $(M,Z,g)$. The critical surface $Z$ may consist of several connected components. Our first main result shows that a $b$-Beltrami field exhibits escape orbits on a generic asymptotically exact $b$-manifold (see Definition~\ref{def:bexactmetric}). The meaning of generic will become clear in Section~\ref{sec:mainresult}.

%\begin{theorem}\label{thm:mainthm}
%There exists a residual set $\widehat {\mathcal G_b^k}\subset \mathcal G^k_b$ such that any $b$-Beltrami vector field on $(M,Z,g)$, $g\in\widehat {\mathcal  G_b^k}$, which is not identically zero on $Z$, has a singular periodic orbit or at least $2+b_1(Z)$ escape orbits. Here, $b_1(Z)$ denotes the first Betti number of the critical surface $Z$.
%\end{theorem}

\begin{theorem}\label{thm:mainthmreformulation}
Let $g$ be a generic metric in the class of asymptotically exact $b$-metrics on $(M,Z)$.
Then any $b$-Beltrami vector field on $(M,Z,g)$, which is not identically zero on $Z$, has a singular periodic orbit or at least $2+b_1(Z)$ escape orbits. Here, $b_1(Z)$ denotes the first Betti number of the critical surface $Z$.
\end{theorem}

%\begin{remark}
%The term \emph{residual} means that $\widehat{\mathcal G_b^k}$ is a countable intersection of open and dense subsets of $\mathcal G_b^k$; in particular, $\widehat{\mathcal G_b^k}$ is dense in $\mathcal G_b^k$.
%\end{remark}

In light of the Beltrami/contact correspondence, this result has a clear implication in the $b$-contact context: the Reeb flow of a generic Melrose $b$-contact form has a singular periodic orbit or at least $2+b_1(Z)$ escape orbits. For the concept of \emph{generic Melrose $b$-contact form} see Definition~\ref{def:bMelrose}.

The proof of Theorem~\ref{thm:mainthmreformulation} lies in a careful study of the $b$-Beltrami vector field on $Z$. As is proved in~\cite{MO2}, the $b$-Beltrami vector field on $Z$ is Hamiltonian. Additionally, the Beltrami equation yields that the Hamiltonian function, called \emph{exceptional Hamiltonian} (see Definition \ref{def:exceptionalHam}) is in fact an eigenfunction of the induced Laplacian on the critical surface. By the classical work of Uhlenbeck~\cite{uhlenbeck}, the properties of the eigenfunctions of the Laplacian for a generic set of metrics are well understood: they are Morse and zero is a regular value. Theorem~\ref{thm:mainthmreformulation} then follows from this result and a local analysis of the zeros of the $b$-Beltrami vector field on $Z$.

The way a singular periodic orbit escapes to infinity is very particular, both its $\omega$- and $\alpha$-limit sets consist of a single point $p_\pm\in Z$. Since the limit sets of a vector field are invariant, the points $p_\pm$ are necessarily zeros of the field. One can relax this condition and introduce the notion of \emph{generalized singular periodic orbit}:

\begin{definition}\label{def:gspo}
Let $\gamma:\R\to M \setminus Z$ be an orbit of a $b$-Beltrami vector field. We say it is a generalized singular periodic orbit if there exist $t_1<t_2<\cdots<t_k\to \infty$ such that $\gamma(t_k)\to p_+\in Z$ and $t_{-1}>t_{-2}>\cdots >t_{-k}\to -\infty$ such that $\gamma(t_{-k})\to p_-\in Z$, as $k\to\infty$. In general, $p_+$ and $p_-$ may be contained in different components of $Z$, and they do not need to be zeros of the field.
\end{definition}

Equivalently, $\gamma$ is a generalized singular periodic orbit if both its $\alpha$- and $\omega$-limit sets have nonempty intersection with the critical surface $Z$. Particular cases of generalized singular periodic orbits that have attracted considerable attention are the \emph{oscillatory motions}, see~\cite{s,llibre,GMS} and Section~\ref{sec:concludingrmks}. The $b$-analogue of the $ABC$ Beltrami vector fields introduced in Example~\ref{ex:T3flat} also exhibits generalized singular periodic orbits:

\begin{example}\label{ex:gspo}
Consider the $b$-Beltrami field presented in Example~\ref{ex:T3flat} on the flat $b$-manifold $(\mathbb{T}^3,Z)$, and let $H$ be the Hamiltonian function as before and $c>0$ a regular value of $H$. Consider the integral curve $\gamma(t)$ of the point $(x_0,y_0,z_0)$, with $z_0\in (0,\pi)$ and $(x_0,y_0)\in H^{-1}(c)$. It is easy to check that both the $\alpha$- and $\omega$-limit of $\gamma(t)$ are periodic orbits that are contained in the components $\{z=\pi\}$ and $\{z=0\}$, respectively, of the critical set $Z$; accordingly, $\gamma(t)$ is a generalized singular periodic orbit that is not a singular periodic orbit.
\end{example}

Our second main result shows the generic existence of generalized singular periodic orbits on asymptotically flat $b$-manifolds (see Definition~\ref{def_flat}) for a particularly relevant class of $b$-Beltrami vector fields. Since the metric is fixed now in a neighborhood of the critical surface $Z$ (a flat $b$-metric), the statement is different from Theorem~\ref{thm:mainthmreformulation}, where the asymptotically exact $b$-metric has to be perturbed. Instead, we consider a generic set of asymptotically symmetric $b$-Beltrami vector fields on the manifold, where \emph{generic} means that the set is open and dense in the $C^k$-topology of vector fields.

\begin{theorem}\label{thm:mainthm2ref}
A generic asymptotically symmetric $b$-Beltrami vector field $X$ on an asymptotically flat $b$-manifold of dimension $3$ has a generalized singular periodic orbit. Moreover, it has a singular periodic orbit or at least $4$ escape orbits.
\end{theorem}
%\begin{remark}
%We will introduce \emph{condition $A$} in Definition~\ref{def:condA}; for the moment we just notice that this condition is satisfied by a nonempty open set of $b$-Beltrami vector fields on a flat $b$-manifold.
%\end{remark}

The proof of the second part in Theorem~\ref{thm:mainthm2ref} is also based on the aforementioned connection between the exceptional Hamiltonian of the $b$-Beltrami vector field and the eigenfunctions of the Laplacian on the critical surface (in this case the flat torus). Since the metric in a neighborhood of $Z$ is fixed now, we have to use a different genericity technique (see Lemma~\ref{L:genflat}). The first part of Theorem~\ref{thm:mainthm2ref} follows from a study of the global minimum of the exceptional Hamiltonian on $Z$ combined with an argument that exploits that a $b$-Beltrami field preserves a smooth volume form in $M\backslash Z$. Furthermore, under the same conditions as in Theorem~\ref{thm:mainthm2ref}, we shall prove (cf. Proposition~\ref{prop:genericsingularweinstein}) that if the field is globally symmetric and the $b$-metric is globally flat, the singular Weinstein conjecture holds.
	
\subsection*{Organization of the article}
After the introduction, we start reviewing the necessary results of $b$-contact manifolds and the relation to $b$-Beltrami vector fields in Section \ref{sec:preliminiaries}. A review on Uhlenbeck's theory of the generic properties of eigenfunctions of the Laplacian can be found in the same section. The proof of the two main theorems (Theorems~\ref{thm:mainthmreformulation} and~\ref{thm:mainthm2ref}) can be found in Sections~\ref{sec:mainresult} and~\ref{subsec:flat}, respectively. We end this article with a discussion of related results and future lines of research in Section~\ref{sec:concludingrmks}.
	
\section{Preliminaries}\label{sec:preliminiaries}
	
In this section, we review basic concepts and results on $b$-contact manifolds from~\cite{MO1,MO2}, as well as their connection with  $b$-Beltrami vector fields, as proved in~\cite{CMP}. We also include a review of the classical theory of Uhlenbeck~\cite{uhlenbeck} of generic eigenfunctions of the Laplacian on compact manifolds.
	
\subsection{Review of $b$-contact manifolds}
	
Let $Z\subset M$ be a smooth hypersurface, called critical hypersurface (possibly disconnected). We assume that $Z$ is the regular zero-level set of a globally defined function $z$. The $b$-tangent bundle consists in the vector bundle whose sections give rise to vector fields that are tangent to $Z$. This vector bundle is denoted by ${^b}TM$ and its dual by ${^b}T^*M$.

{As proved in \cite{GMP} in Section $3$, the restriction of sections of the $b$-tangent bundle to $Z$ is induced by an injective vector bundle morphism
$${^b}TM \to TZ$$
whose kernel is a $1$-dimensional line bundle with canonical non-vanishing section. This vector field is known as \emph{normal $b$-vector field}. In terms of the defining function $z$, the normal $b$-vector field is given by $z\partial_z$.}

The language of differential forms for the $b$-tangent bundle was introduced by Melrose \cite{melrose} and further elaborated in \cite{GMP}.
	
\begin{definition}
The differential form $\omega\in \Gamma(\bigwedge^k {^b}T^*M)$ is called a \emph{$b$-form} of degree $k$.
\end{definition}
	
For $b$-forms, the following holds.
	
\begin{lemma}[\cite{GMP}]\label{decomposition}
Let $\omega \in {^b}\Omega^k(M)$ be a $b$-form of degree $k$. Then $\omega$ decomposes as follows:
$$ \omega = \frac{dz}{z}\wedge \alpha +\beta, \quad \alpha \in \Omega^{k-1}(M),\ \beta \in \Omega^k(M).$$
\end{lemma}
The exterior derivative for smooth differential forms can be extended by defining
$$d\omega:= \frac{dz}{z} \wedge d\alpha +d\beta.$$
	
\begin{remark} The restriction of the
closed form $d\alpha$ to the surface $z=0$ is often called the residue of
$d\omega$ as it was done in \cite{geoffandeva} and it is customary in the works of polar homology, see \cite{boris}.
\end{remark}
\begin{definition}
Let $(M,Z)$ be a (2n+1)-dimensional $b$-manifold. A $b$-contact form is a $b$-form of degree one $\alpha \in {^b\Omega^1(M)}$, that satisfies $\alpha \wedge (d\alpha)^n \neq 0$ as a section of $\Lambda^{2n+1}(^bT^*M)$. We say that the pair $(M,\alpha)$ is a $b$-contact manifold.
\end{definition}
	
Associated to a $b$-contact form, there exists a unique vector field, that is tangent to the critical hypersurface, called the Reeb vector field, defined by the equations
$$\begin{cases} \iota_{R_\alpha}d\alpha=0 \\ \iota_{R_\alpha}\alpha=1. \end{cases}$$
It turns out that the Reeb vector field restricted to $Z$ is a Hamiltonian vector field:

\begin{proposition}[\cite{MO1}]\label{prop:ReebHamdim3}
Let $(M,\alpha=u\frac{dz}{z}+\beta)$ be a $b$-contact manifold of dimension $3$, where $u\in C^\infty(M)$ and $\beta\in \Omega^1(M)$ as in Lemma~\ref{decomposition}. Then the restriction on $Z$ of the $2$-form $\Theta= u d\beta+\beta\wedge du$ is symplectic and the Reeb vector field is Hamiltonian with respect to $\Theta$ with Hamiltonian function $-u$, i.e., $\iota_{R_\alpha} \Theta= du$.
\end{proposition}
	
The Hamiltonian function $-u$ plays a key role in what follows in this paper and therefore we introduce the following definition:
	
\begin{definition}\label{def:exceptionalHam}
The Hamiltonian function $-u|_Z$ associated to a $b$-contact manifold $(M,\alpha)$ is called \emph{exceptional Hamiltonian}.
\end{definition}
	
In \cite{MO2}, the Reeb dynamics on $b$-contact manifolds is studied. Among other results, the authors proved:
	
\begin{itemize}
\item In dimension $3$, there exist infinitely many periodic orbits on $Z$ provided that $Z$ is compact.
\item There exist compact examples in any dimension, without any periodic orbits away from $Z$.
\end{itemize}
	
It follows immediately from these results that the statement on the existence of periodic orbits away from $Z$ needs to be refined. Instead of asking for periodic Reeb orbits, the authors conjectured, cf. Conjecture~\ref{conj}, the existence of \emph{singular periodic Reeb orbits}.
	
\begin{definition}\label{def:spo}
Let $(M,\alpha)$ be a $b$-contact manifold with critical hypersurface $Z$. A \emph{singular periodic orbit} $\gamma:\R \to M \setminus Z$ is an integral curve of the Reeb field such that $\lim_{t \to  \pm \infty} \gamma(t) =p_{\pm} \in Z$ where $R_{\alpha}(p_{\pm})=0$.
\end{definition}
	
As explained in Section~\ref{S:intro}, singular periodic orbits are a particular case of escape orbits, which play a fundamental role in this work.
	
\subsection{$b$-Beltrami vector fields}\label{SS.bBelt}
	
The motivation to study $b$-manifolds comes from manifolds with cylindrical ends: a Riemannian metric on the manifold with cylindrical ends then translated into a Riemannian metric for the $b$-tangent bundle.
	
\begin{definition}[$b$-metrics]
A $b$-metric is a bilinear positive-definite form $\Gamma({^b}T^*M \otimes {^b}T^*M)$ and $(M,Z,g)$ is called a $b$-Riemannian manifold. The $b$-metric naturally induces a $b$-form of maximal degree that is called $b$-volume form.
\end{definition}

Beltrami vector fields are a special class of solutions to the stationary Euler equations on a Riemannian $3$-manifold. They are defined as vector-valued eigenfunctions of the curl operator. Following~\cite{CMP}, Beltrami vector fields on a manifold with cylindrical ends can be modelled using a $b$-metric on a $b$-manifold $(M,Z)$. This gives rise to the concept of $b$-Beltrami vector field:

\begin{definition}
A $b$-Beltrami vector field $X$ is a vector field on a Riemannian $b$-manifold $(M,Z,g)$ such that $\curl X=\lambda X$, for some nonzero constant $\lambda$, where the curl operator is defined with respect to the $b$-metric $g$.
\end{definition}
	
Similar to the smooth case, any non-vanishing $b$-Beltrami vector field is a reparametrization of the Reeb field associated to a $b$-contact form. More precisely, we have the following:
	
\begin{theorem}[\cite{CMP}]\label{thm:bBeltramicorrespondence}
Let $(M,Z)$ be a $b$-manifold of dimension three. Any $b$-Beltrami vector field that is non-vanishing as a section of ${^b}TM$ on $M$ is a Reeb field (up to rescaling) for some $b$-contact form on $(M,Z)$. Conversely, given a $b$-contact form $\alpha$ with Reeb field $X$ then any nonzero rescaling of $X$ is a $b$-Beltrami vector field for some $b$-metric and $b$-volume form on $M$.
\end{theorem}
	
\begin{remark}\label{rem:bBeltramicorrespondence}
It follows from the proof of Theorem~\ref{thm:bBeltramicorrespondence} that if $X$ is a $b$-Beltrami vector field on $(M,Z,g)$, the Reeb field associated to the $b$-contact form $\alpha:=g(X,\cdot)$ is given by $\frac{1}{\norm{X}^2}X$, where the norm is computed using the $b$-metric $g$.
\end{remark}
	
In this paper we shall consider a specially relevant class of metrics, that are known as asymptotically exact $b$-metrics. To define them, let us consider a tubular neighbourhood $\mathcal{N}(Z)$ around the critical surface and the trivialization (a diffeomorphism) given by
$$(z,P): \mathcal{N}(Z) \to (-\epsilon,\epsilon) \times Z\,.$$
Following Section 2.2 in~\cite{melrose}, the \emph{asymptotically exact $b$-metrics} are a special class of metrics on the $b$-tangent bundle:

\begin{definition}\label{def:bexactmetric}
An asymptotically exact $b$-metric is a metric $g$ on ${^b}TM$ which can be written in a neighborhood $\mathcal N(Z)$, in terms of the aforementioned trivialization, as
\begin{equation}\label{eq:exactbmetric}
g= \frac{dz^2}{z^2}+P^*h\,,
\end{equation}
where $h$ is a smooth Riemannian metric on $Z$. The space of asymptotically exact $b$-metrics of class $C^k$ on $(M,Z)$ is denoted by $\mathcal G_b^k$; in the neighborhood $\mathcal N(Z)$ it inherits via the map $P$ the $C^k$ topology of the space of $C^k$ Riemannian metrics on $Z$.
\end{definition}

The Beltrami/contact correspondence shows that a non-vanishing $b$-Beltrami vector field on an asymptotically exact $b$-manifold $(M,Z,g)$ defines a $b$-contact form. This allows us to introduce the concept of Melrose $b$-contact form, which is the class of $b$-contact forms that we consider in this article:

\begin{definition}\label{def:bMelrose}
The $b$-contact forms obtained via the correspondence Theorem~\ref{thm:bBeltramicorrespondence} from a nonvanishing $b$-Beltrami vector field on an asymptotically exact $b$-manifold are called \emph{Melrose $b$-contact forms}.
\end{definition}

We also observe that Theorem~\ref{thm:bBeltramicorrespondence} and Proposition~\ref{prop:ReebHamdim3} imply that the exceptional Hamiltonian can be read from a $b$-Beltrami vector field on an asymptotically exact $b$-manifold. More precisely, consider local coordinates $(x,y,z)$ associated to the metric splitting~\eqref{eq:exactbmetric}, with $(x,y)\in Z$ and $z$ the defining function (i.e., $Z=\{z=0\}$):
\[
g=\frac{dz^2}{z^2}+h_{11}dx^2+h_{22}dy^2+2h_{12}dxdy\,,
\]
where $h_{ij}(x,y)$ defines a Riemannian metric on $Z$. In these coordinates, a $b$-Beltrami vector field reads as:
\begin{equation}\label{Eq:bBelt}
X=X_x\partial_x+X_y\partial_y+zX_z\partial_z\,.
\end{equation}
We claim that the restriction $X|_Z$ is a Hamiltonian field with Hamiltonian function $-X_z$ (the exceptional Hamiltonian). Indeed, since the correspondence Theorem~\ref{thm:bBeltramicorrespondence} implies that $X$ is a reparametrization of the Reeb field associated to the $b$-contact form
$$\alpha=X^\flat=\frac{X_zdz}{z}+(h_{11}X_x+h_{12}X_y)dx+(h_{12}X_x+h_{22}X_y)dy\,,$$
then it follows from Proposition~\ref{prop:ReebHamdim3} that $-X_z$ is the exceptional Hamiltonian and $X|_Z$ is the corresponding Hamiltonian vector field (after rescaling of the symplectic form due to the reparametrization factor). We have then established the following:

\begin{lemma}\label{L.excHam}
Let $X$ be a $b$-Beltrami vector field on an asymptotically exact $b$-manifold. Then $X|_{Z}$ is a Hamiltonian vector field for some symplectic form on $Z$, whose corresponding exceptional Hamiltonian is~$-X_z$.
\end{lemma}

Let us illustrate this result with the following example of a $b$-Beltrami vector field on a globally flat $b$-manifold:
	
\begin{example}\label{ex:abcfield}
Consider $X=C\cos y \partial_x +B \sin x \partial_y + (C\sin y + B \cos x)z \partial_z$ on $\mathbb{T}^2\times \R$, which is a $b$-Beltrami vector field for the flat $b$-metric $g=\frac{dz^2}{z^2}+dx^2+dy^2$, $(x,y)\in\mathbb T^2$, $z\in \mathbb R$. By Lemma~\ref{L.excHam}, it is Hamiltonian on the critical set $Z=\{z=0\}$ and the exceptional Hamiltonian is given by $H= -C\sin y -B\cos x$. This is an eigenfunction of the Laplacian $\Delta=\partial^2_{xx}+\partial^2_{yy}$ on the flat $\mathbb{T}^2$ with eigenvalue $1$. We shall establish in Section~\ref{sec:mainresult} that this is, in fact, a general property of the exceptional Hamiltonians.
\end{example}

\subsection{Generic eigenfunctions of the Laplacian}\label{subsec:uhlenbeck}
	
Let $N$ be a compact manifold (without boundary) and denote by $\mathcal{G}^k$ the space of Riemannian metrics on $N$ of class $C^k$, $k\geq 2$. It is well known that $\mathcal G^k$ is a Banach manifold. Given a metric $h\in \mathcal G^k$, the Laplacian $\Delta_h$ (or Laplace-Beltrami operator) is a second-order elliptic operator with coefficients of class $C^{k-1}$. The corresponding eigenfuntions $u_k$ satisfy the equation
\[
-\Delta_hu_k=\lambda_k u_k\,,
\]
where $0=\lambda_0<\lambda_1\leq \lambda_2\leq \ldots$ are the eigenvalues. Standard regularity estimates imply that the eigenfunctions are of class $C^{k,\alpha}$ for all $\alpha<1$.
	
A landmark in spectral geometry is Uhlenbeck's work~\cite{uhlenbeck}, where she proved that there exists a residual set of metrics in $\mathcal G^k$ (i.e., a countable intersection of open and dense sets) whose corresponding eigenfunctions and eigenvalues satisfy several nondegeneracy properties. More precisely:

\begin{theorem}[\cite{uhlenbeck}]\label{thm:uhlenbeck}
For all $k\geq 2$, there exists a residual set $\widehat{\mathcal G^k}\subset \mathcal G^k$ of metrics such that the Laplacian $\Delta_h$ has the following properties provided that $h\in \widehat{\mathcal G^k}$:
\begin{enumerate}[label=(\Alph*)]
\item Its spectrum is simple, i.e., all the eigenvalues have multiplicity one.
\item The zero set of all the (nonconstant) eigenfunctions is regular.
\item All the (nonconstant) eigenfunctions are Morse.
\end{enumerate}
\end{theorem}

This result will be key to the proof of the first main theorem (Theorem~\ref{thm:mainthmreformulation}).

\section{Generic existence of escape orbits}\label{sec:mainresult}

This section contains the proof of the first main theorem. More precisely, we prove the following:

\begin{theorem}[First main theorem]\label{thm:mainthm}
There exists a residual set $\widehat {\mathcal G_b^k}\subset \mathcal G^k_b$ such that any $b$-Beltrami vector field on $(M,Z,g)$, $g\in\widehat {\mathcal  G_b^k}$, which is not identically zero on $Z$, has a singular periodic orbit or at least $2+b_1(Z)$ escape orbits. Here, $b_1(Z)$ denotes the first Betti number of the critical surface $Z$.
\end{theorem}

\begin{remark}
The term \emph{residual} means that $\widehat{\mathcal G_b^k}$ is a countable intersection of open and dense subsets of $\mathcal G_b^k$; in particular, $\widehat{\mathcal G_b^k}$ is dense in $\mathcal G_b^k$. This is what we mean by generic metric in Theorem~\ref{thm:mainthmreformulation}.
\end{remark}

We first establish a remarkable connection between $b$-Beltrami vector fields and spectral geometry on the critical surface $Z$. Specifically, we prove that
the exceptional Hamiltonian associated with a $b$-Beltrami vector field on an asymptotically exact $b$-manifold is an eigenfunction of the Laplacian on $Z$. In the language of $b$-contact geometry, this implies that the exceptional Hamiltonian for Melrose $b$-contact forms is an eigenfunction of the Laplacian. See Section~\ref{SS.bBelt} for the corresponding definitions. This connection will allow us to invoke the machinery of spectral theory (Uhlenbeck's theorem, in particular) to study generic dynamics in $b$-contact geometry.

\begin{proposition}\label{prop:eigenfunctionlaplacian}
Let $X$ be a $b$-Betrami field on an asymptotically exact $b$-manifold $(M,Z,g)$. Then the exceptional Hamiltonian is an eigenfunction of the Laplacian $\Delta_h$.
\end{proposition}

\begin{proof}

Let us consider the local coordinates $(x,y,z)\in Z\times (-\epsilon, \epsilon)$ introduced in Equation~\eqref{Eq:bBelt} to write a $b$-Beltrami vector field $X$. In these coordinates the $b$-metric $g$ reads in a tubular neighborhood of $Z$ as
\[
g=h(x,y)+\frac{dz^2}{z^2}
\]
where $h(x,y)=h_1dx^2+h_2dy^2+2h_{12}dxdy$ is a Riemannian metric on $Z$, cf. Equation~\eqref{eq:exactbmetric}. According to Lemma~\ref{L.excHam}, the exceptional Hamiltonian is given by $-X_z$ on $Z$, i.e., in these coordinates $-X_z(x,y,0)$, so let us compute the Laplacian $\Delta_h X_z(x,y,0)$.

To this end, we notice that the induced $b$-volume form is given by $\mu_g=\sqrt{\det h}\,dx\wedge dy \wedge \frac{dz}{z}$, and the $b$-form $\alpha$ dual to $X$ computed with the $b$-metric is:
\[
\alpha=\Big(h_{11}X_x+h_{12}X_y\Big)dx+\Big(h_{12}X_x+h_{22}X_y\Big)dy+\frac{X_z}{z}\,dz\,.
\]
In terms of $\alpha$, the $b$-Beltrami equation has the expression $\lambda \iota_X \mu_g=d\alpha$, which reads in coordinates as

\begin{equation}\label{eq:beltramisplitmetric}
\begin{cases}  \lambda \sqrt{\det h} X_z=-\partial_y(h_{11}X_x+h_{12}X_y)+\partial_x(h_{12}X_x+h_{22}X_y)\,, \\ -\lambda \sqrt{\det h} X_y=\partial_xX_z-z\partial_z(h_{11}X_x+h_{12}X_y)\,, \\ \lambda \sqrt{\det h}X_x=\partial_yX_z-z\partial_z(h_{12}X_x+h_{22}X_y)\,.\end{cases}
\end{equation}

Therefore, restricting on the critical set $\{z=0\}$ we obtain

\begin{equation}\label{eq:beltramisplitmetric}
\begin{cases}  \lambda \sqrt{\det h} X_z=-\partial_y(h_{11}X_x+h_{12}X_y)+\partial_x(h_{22}X_y+h_{12}X_x) \\ -\lambda \sqrt{\det h} X_y=\partial_xX_z \\ \lambda \sqrt{\det h}X_x=\partial_yX_z\,,\end{cases}
\end{equation}
where all the function $X_k$ are evaluated at $(x,y,0)$. Finally, noticing that the Laplacian $\Delta_h$ on $Z$ is given in local coordinates $(x,y)$ by
$$\Delta_h=\frac{1}{\sqrt{\det h}} \bigg[ \frac{\partial }{\partial x}\bigg(\frac{1}{\sqrt{\det h}}h_{22} \partial_x-\frac{1}{\sqrt{\det h}}h_{12} \partial_y\bigg)+\partial_y\bigg(\frac{1}{\sqrt{\det h}}h_{11} \partial_y-\frac{1}{\sqrt{\det h}}h_{12} \partial_x\bigg) \bigg]\,,$$
it readily follows from the system of equations~\eqref{eq:beltramisplitmetric} that
$$\Delta_hX_z=-\lambda^2X_z\,,$$
thus proving that the exceptional Hamiltonian $-X_z$ is an eigenfunction of the Laplacian $\Delta_h$ with eigenvalue $\lambda^2$. The proposition then follows.
\end{proof}

As reviewed in Subsection~\ref{subsec:uhlenbeck}, Uhlenbeck's theory characterizes the critical set of the eigenfunctions of the Laplacian for a generic set of metrics (i.e., a residual set), cf. Theorem~\ref{thm:uhlenbeck}. This result is very useful when studying asymptotically exact $b$-metrics because the $b$-metric $g$ is written (in a neighborhood of $Z$) in terms of a Riemannian metric $h$ on the critical surface. Accordingly, a residual set of metrics $\widehat{\mathcal G^k}$ on $Z$ endowed with the $C^k$-topology obviously defines a set of $b$-metrics $\widehat{\mathcal G_b^k}$ on $(M,Z)$ via the splitting~\eqref{eq:exactbmetric}, which is also residual with the $C^k$-topology in the set of all asymptotically exact $b$-metrics $\mathcal G_b^k$.

In view of Definition~\ref{def:bMelrose}, by generic Melrose $b$-contact form, we mean the $b$-contact form associated with a nonvanishing $b$-Beltrami vector field on $(M,Z,g)$, with $g\in \widehat{\mathcal G_b^k}$. As a corollary of Proposition~\ref{prop:eigenfunctionlaplacian} and Theorem~\ref{thm:uhlenbeck}, we then obtain the following:

\begin{corollary}\label{thm:genericbBeltrami}
The (nonconstant) exceptional Hamiltonian associated with a generic $b$-Beltrami vector field or a generic Melrose $b$-contact form is a Morse function on $Z$ and its zero set is regular.
\end{corollary}

We are now ready to prove Theorem~\ref{thm:mainthm}. We recall that the first Betti number of a compact surface $Z$ is given by $b_1(Z)=2\nu(Z)$, where $\nu(Z)$ is the genus of the surface. In the proof it is key to use that the exceptional Hamiltonian is a Morse function with regular zero set to prove the existence of stable or unstable directions which, in turn, correspond to escape orbits. The role of $b_1(Z)$ is that it bounds from below the number of critical points of a Morse function on $Z$.

\begin{proof}[Proof of Theorem~\ref{thm:mainthm}]
Let $X$ be a $b$-Beltrami vector field on an asymptotically exact $b$-manifold $(M,Z,g)$. In what follows we use the same local coordinates $(x,y,z)$ introduced above. Since $X|_Z$ is tangent to the critical surface $Z$ and Hamiltonian with Hamiltonian function given by $H(x,y):=-X_z(x,y,0)$, the zeros of $X$ on $Z$ are given by the critical points of $H$. Now let us analyze the linear stability of these zeros. To this end, we compute the Jacobian matrix $DX$ at a critical point $p=(x_0,y_0,0)$ of $H$. Using Equations~\eqref{eq:beltramisplitmetric}, a straightforward computation shows that
\begin{equation}\label{eq:JacobianBeltrami}
DX(p)=\frac{1}{\lambda \sqrt{\det h}}\begin{pmatrix}
-\partial^2_{xy}H & -\partial^2_{yy}H & * \\
\partial^2_{xx}H & \partial^2_{xy}H & *\\
0 & 0 &  -\lambda\sqrt{\det h}\,H
\end{pmatrix}\Bigg|_{p}.
\end{equation}
In the two first rows, we used that $\partial_xH(p)=\partial_yH(p)=0$ because $p$ is a critical point of $H$. For example, to compute the first entry of the matrix we write
$$\partial_x X_x(p)=-\partial_x\bigg( \frac{1}{\lambda \sqrt{ \det h}} \partial_y H\bigg)\bigg|_{p}=-\partial_x\bigg(\frac{1}{\lambda \sqrt{\det h}} \bigg)\partial_y H\bigg|_{p}-\frac{1}{\lambda \sqrt{\det h}}{\partial^2_{xy} H}\bigg|_{p}\,,$$
which at $p$ therefore gives $-\frac{1}{\lambda \sqrt{\det h}}\partial^2_{xy}H(p)$.
In the last row, we used that the computation is carried out on $Z$ and that $\partial_z (zH)|_Z=H$.

If $X|_Z$ is not identically zero, and the asymptotically exact $b$-metric is generic in the sense that $g\in \widehat{\mathcal G_b^k}$, Corollary~\ref{thm:genericbBeltrami} implies that $H$ is a Morse function with regular zero set. Therefore, we deduce from Morse inequalities that there are at least $2+b_1(Z)$ critical points $p_k$. It also follows that the matrix $DX(p_k)$ is non-singular at each critical point $p_k$: indeed the determinant of the matrix in Equation~\eqref{eq:JacobianBeltrami} is given by
$$\det DX(p_k)=\frac{-1}{\lambda \det h(p_k)} \Hess H (p_k) H(p_k) \neq 0\,,$$
where $\Hess$ is the Hessian determinant, and we have used that $\Hess H(p_k)\neq 0 \neq H(p_k)$ for all the critical points $p_k$. Moreover, $DX(p_k)$ has three nonzero eigenvalues $\lambda_x,\lambda_y,\lambda_z$, such that $\lambda_x$ and $\lambda_y$ are the eigenvalues of the $2\times 2$ matrix
\begin{equation}\label{eq:JacobianBeltrami}
\frac{1}{\lambda \sqrt{\det h}}\begin{pmatrix}
-\partial^2_{xy}H & -\partial^2_{yy}\\
\partial^2_{xx}H & \partial^2_{xy}H
\end{pmatrix}\Bigg|_{p_k}\,,
\end{equation}
and $\lambda_z=-H(p_k)$. Now there are four cases to discuss:
\begin{enumerate}
\item $H(p_k)>0$ and $p_k$ is a saddle point. In this case, the $b$-Beltrami vector field $X$ has a $2$-dimensional stable manifold at $p_k$ that is transverse to $Z$.
\item $H(p_k)<0$ and $p_k$ is a saddle point. In this case, the $b$-Beltrami vector field $X$ has a $2$-dimensional unstable manifold at $p_k$ that is transverse to $Z$.
\item $H(p_k)>0$ and $p_k$ is a local maximum or minimum. In this case, the $b$-Beltrami vector field $X$ has a $1$-dimensional stable manifold at $p_k$ that is transverse to $Z$.
\item $H(p_k)<0$ and $p_k$ is a local maximum or minimum. In this case, the $b$-Beltrami vector field $X$ has a $1$-dimensional unstable manifold at $p_k$ that is transverse to $Z$.
\end{enumerate}
In all these cases, which follow from an easy application of the invariant manifold theorem, the vector field $X$ has transverse invariant manifolds of dimension $1$ or $2$ whose $\omega$- or $\alpha$-limit are a point $p_k\in Z$. In particular, there exists at least one escape orbit for each $p_k$. Finally, if two escape orbits corresponding to different points $p_{k_1}\neq p_{k_2}$ coincide, this gives, by definition, a singular periodic orbit, cf. Definition~\ref{def:spo}. If escape orbits do not coincide, there are at least as many as critical points $p_k$ on $Z$, a number that we know it is lower bounded by $2+b_1(Z)$. The theorem then follows.
\end{proof}

\begin{remark}
In cases $(1)$ and $(2)$ discussed above, the set of escape orbits is actually $2$-dimensional. At least one of these cases occur whenever the critical surface $Z$ has positive genus.
\end{remark}
\begin{remark}\label{rmk:generalcase}
Theorem~\ref{thm:mainthm} is not enough to prove the singular Weinstein conjecture for Melrose $b$-contact forms. The key is that we do not prove that the stable (respectively unstable) manifold of the equilibrium point $p_{k_1}$ intersects with the unstable (respectively stable) manifold of another equilibrium point $p_{k_2}$. We observe that for more general $b$-metrics (that are not exact), the exceptional Hamiltonian is an eigenfunction of a more complicated elliptic operator (not a Laplacian). We were not able to show the genericity properties à la Uhlenbeck for eigenfunctions of those operators.
\end{remark}

\begin{figure}[hbt!]
\begin{center}
\begin{tikzpicture}[scale=0.75]
%\fill[gray!20] (-1.5,-0.3) rectangle +(3,0.6);
\draw[rotate=45,color=red,dashed] (4,-0.5) ellipse (1cm and 0.5cm);
%\fill [shift={(3.4,2)},scale=0.1,rotate=210]   (0,0) -- (-1,-0.7) -- (-1,0.7) -- cycle;
\draw[red](1.8,2) node {$Z_1$};
 \draw[dashed,color=red] (-3,0) --  (6,0);
 \draw[red] (5.2,-0.5) node {$Z_2$};
 \draw (0.6,1) node {$\gamma_1$};
 \fill[green] (0,0) circle[radius=3pt];
 \draw (0,0) ..controls +(0.2,2) and +(-0.5,-0.3).. (2.5,3.2);
 \draw (2.5,3.2) ..controls +(0.4,0.2) and +(-0.25,0.25).. (4,3.5);
 \draw (4,3.5) ..controls +(0.4,-0.4) and +(0.2,0.2).. (3.75,2);
 \draw (3.75,2) ..controls +(-0.3,-0.3) and +(0.4,-0.35).. (2.25,1.55);
 \draw (2.25,1.55) ..controls +(-0.3,0.6) and +(-0.5,-0.35).. (3,3.2);
 \draw[dashed] (3,3.2) ..controls +(0.25,0.15) and +(-0.3,0.35).. (4.03,3.2);
 \fill [shift={(0.68,1.7)},scale=0.1,rotate=50]   (0,0) -- (-1,-0.7) -- (-1,0.7) -- cycle;
  \draw[dashed,color=red] (-3,-2.5) --  (6,-2.5);
  \draw[red] (5.2,-2.2) node {$Z_3$};
    \draw[dashed,color=red] (-3,-4.5) --  (6,-4.5);     \fill[green] (-1,-2.5) circle[radius=3pt];
    \fill[green] (-2,0) circle[radius=3pt];
     \draw (-2,0) ..controls +(-0.3,-0.6) and +(-0.1,1).. (-0.5,-1);
    \draw (-0.5,-1) ..controls +(0,-1) and +(-0.1,0.35).. (-1,-2.5);
    \fill [shift={(-0.5,-1.25)},scale=0.1,rotate=80]   (0,0) -- (-1,-0.7) -- (-1,0.7) -- cycle;
  \draw[red] (5.2,-4.2) node {$Z_4$};
  \fill[green] (1,-4.5) circle[radius=3pt];
    \fill[green] (3,-4.5) circle[radius=3pt];
     \draw (1,-4.5) ..controls +(-0.3,-0.6) and +(-0.5,0.5).. (2,-3.8);
    \draw (2,-3.8) ..controls +(0.5,-0.5) and +(-0.1,0.35).. (3,-4.5);
     \draw (-0.2,-2) node {$\gamma_2$};
     \draw (1,-3.8) node {$\gamma_3$};
    \fill [shift={(1.8,-3.72)},scale=0.1,rotate=0]   (0,0) -- (-1,-0.7) -- (-1,0.7) -- cycle;
\end{tikzpicture}
\caption{Different types of escape and singular periodic orbits: $\gamma_1$ is a generalized singular periodic orbit, $\gamma_2,\gamma_3$ are singular periodic orbits}
\label{fig:singularvsescape}
\end{center}
\end{figure}
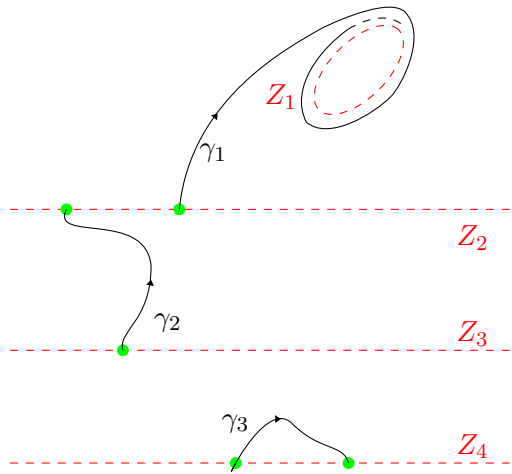

\section{Asymptotically flat $b$-metrics}\label{subsec:flat}

This section contains the proof of the second main theorem. More precisely, we prove the following:

\begin{theorem}[Second main theorem]\label{thm:mainthm2}
A generic asymptotically symmetric $b$-Beltrami vector field $X$ on an asymptotically flat $b$-manifold of dimension $3$ has a generalized singular periodic orbit. Moreover, it has a singular periodic orbit or at least $4$ escape orbits.
\end{theorem}

A particularly relevant model of $b$-manifold is the flat one, which is the $b$-analogue of the flat torus in Riemannian geometry. It is then natural to analyze the dynamics of $b$-Beltrami vector fields on asymptotically flat $b$-manifolds, as e.g. in Example~\ref{ex:T3flat} where the ABC flows in the $b$-context are introduced. We will consider the following definition:

\begin{definition}\label{def_flat}
An asymptotically flat $b$-manifold $(M,Z,g)$ of dimension $3$ is a manifold whose critical surface $Z$ consists of several (disjoint) copies of $\mathbb T^2$, endowed with a $b$-metric which has the following expression in a neighborhood $\mathcal N(Z)$ of $Z$:
\[
g=\frac{dz^2}{z^2}+dx^2+dy^2
\]
$z\in(-\epsilon,\epsilon)$ and $(x,y)\in\mathbb T^2=(\mathbb R/2\pi\mathbb Z)^2$. The manifold $M$ is globally $b$-flat if $M=\mathbb{T}^3$ and $g=dx^2+dy^2+\frac{dz^2}{sin^2 z}$.
\end{definition}

\begin{example}
An example of a globally flat $b$-metric is given in Example~\ref{ex:T3flat}. The $ABC$ $b$-Beltrami vector fields studied in this example (and in Example~\ref{ex:gspo}) exhibit both singular periodic orbits and generalized singular periodic orbits.
\end{example}

Our goal in this section is to prove the existence of escape orbits (or more generally, of generalized singular periodic orbits as introduced in Definition~\ref{def:gspo}) for $b$-Beltrami vector fields on an asymptotically flat manifold. Contrary to Example \ref{ex:T3flat}, where the existence of singular periodic orbits is proved for all values $B,C$ ($|B|\neq|C|$), we will only address the singular Weinstein conjecture for \emph{generic} $b$-Beltrami fields on asymptotically flat $b$-manifolds.

Since now the $b$-metric is fixed in $\mathcal N(Z)$, the notion of genericity is different from the one we considered in Theorem~\ref{thm:mainthm}. Specifically, generic will refer to an open and dense set of the space of $b$-Beltrami vector fields, which is endowed with the $C^k$ topology. To establish this genericity result we will focus on asymptotically symmetric $b$-Beltrami vector fields, which is a natural class to consider in view of Example~\ref{ex:abcfield}:

\begin{definition}\label{def:asymptoticallysymmetric}
Let $X$ be a $b$-Beltrami vector field on an asymptotically flat $b$-manifold $(M,Z,g)$. We say that it is asymptotically symmetric if it commutes with {the normal $b$-vector field in a neighborhood of the critical surface $Z$.}
\end{definition}

\begin{remark}
{As mentioned in the preliminaries, for a fixed defining function $z$ of the critical surface $Z$, the normal $b$-vector field is given by $z\partial_z$, hence Definition \ref{def:asymptoticallysymmetric} asks that locally around $Z$, the vector field $X$ commutes with $z\partial_z$, i.e. $[X,z\partial_z]=0$.}
\end{remark}

Notice that the $b$-Beltrami vector fields considered in Examples~\ref{ex:T3flat} and~\ref{ex:abcfield} are asymptotically symmetric (in fact, they are globally symmetric).

\subsection{A preliminary spectral lemma}

Our strategy of proof of Theorem~\ref{thm:mainthm2} also employs that the exceptional Hamiltonian is an eigenfunction of the Laplacian on the critical surface. It will then be convenient to establish a genericity result in this setting. More precisely, consider the torus $\mathbb{T}^2=(\R/(2\pi \mathbb{Z})^2$ endowed with the flat metric. It is standard that the spectrum of the Laplacian for the flat metric is given by
$$\{\mu_k=|k|^2: (k_1,k_2)\in \mathbb{Z}^2\}$$
and the multiplicity of each eigenvalue is at least $4$ (indeed $(k_1,k_2)$, $(-k_1,k_2)$, $(k_1,-k_2)$ and $(-k_1,-k_2)$ correspond to the same eigenvalue). Let us denote by $\mathcal{E}_\mu$ the eigenspace of the Laplacian with eigenvalue $\mu$.

Our goal is to prove that the eigenfunctions are generically Morse. To prove this, we will use the following parametric transversality theorem:

\begin{theorem}[Parametric transversality theorem, Theorem~6.35 in~\cite{lee}]
Let $N$ and $M$ be two smooth manifolds, $X\subset M$ an embedded submanifold, and $\{F_s\}_{s\in S}$ a smooth family of maps $F_s:N\to M$ with $S$ a smooth manifold of parameters. If the map $F:N\times S \to M$, $F(\cdot,s) :=F_s(\cdot)$ is transverse to $X$, then for almost every $s\in S$ (in the measure-theoretic sense), the map $F_s:N\to M$ is transverse to $X$.
\end{theorem}

\begin{lemma}\label{L:genflat}
For each eigenvalue $\mu>0$, there exists an open and dense set in $\mathcal{E}_\mu$ (in the $C^k$-topology, $k\geq2$) of eigenfunctions that are Morse.
\end{lemma}

\begin{proof}
Let us consider the map
\begin{align*}
F:\mathbb{T}^2\times \R^m &\to \R^2\\
(x,a) &\mapsto F(x,a)=a_1 \nabla f_1(x)+ \cdots +a_m \nabla f_m(x)\,,
\end{align*}
where $f_1,\dots,f_m$ is an $L^2$-orthonormal basis of eigenfunction in $\mathcal{E}_\mu$, $m$ is the multiplicity of the eigenvalue $\mu$, and $x=(x_1,x_2)$ parameterizes $\mathbb T^2$.

For each $a$, the eigenfunction $a_1f_1+ \cdots +a_mf_m$ is Morse if and only if the zero level-set of the map $F_a:=F(\cdot,a)$ is regular. It is then obvious that the lemma follows if we show that this is the case for almost all $a\in \R^m$ (because a set of total measure is dense, and being Morse is an open property; passing from an open and dense set  of $a\in\R^m$ to the $C^k$ topology of eigenfunctions in $\mathcal{E}_\mu$ is of course immediate). To prove this, we use the parametric transversality theorem stated above with $X=\{0\}$ and $S=\R^m$. Hence it suffices to show that $F$ is transverse to $X=\{0\}$.

Let $(x_*,a_*)$ be a point such that $F(x_*,a_*)=0$. The differential of $F$ is given by the following $2\times (m+2)$ matrix:

\begin{equation}
DF=
\begin{bmatrix}
a_1 f_{1x_1 x_1}+\cdots +a_m f_{m x_1 x_1} & a_1 f_{1x_1 x_2}+\cdots +a_m f_{m x_1 x_2} & f_{1x_1} \cdots f_{mx_1} \\
a_1 f_{1x_2 x_1}+\cdots +a_m f_{m x_2 x_1} & a_1 f_{1x_2 x_2}+\cdots +a_m f_{m x_2 x_2} & f_{1x_2} \cdots f_{mx_2}
\end{bmatrix}\,.
\end{equation}

Assume that $\rank DF(x_*,a_*) <2$. This implies that $\nabla f_l(x_*)=b_l \nabla f_1(x_*)$ for $l=2,\dots, m$ and some constants $b_l\in \R$. For convenience, let us set $v:=\nabla f_1(x_*)$, hence $\nabla f_l(x_*)=b_lv$. Accordingly, for any $\mu$-eigenfunction $f\in \mathcal{E}_\mu$, $\nabla f(x_*)=b v$ for some $b\in \R$. By the transitivity of translations on $\mathbb{T}^2$ and the fact that they commute with the Laplacian, it is easy to conclude that for each $f\in \mathcal{E}_\mu$ there exists a function $b$ such that
$$\nabla f(x)=b(x) v$$
for all $x\in \mathbb T^2$. Moreover, since $0=\curl{\nabla f}=\nabla b(x) \times v$, we have that $b$ is a first integral of the constant vector $v^\perp$ and hence $b(x)=B(v_1x_1+v_2x_2)$ where $v=(v_1,v_2)$.

Integrating the vector equation $\nabla f(x)=B(v_1x_1+v_2x_2)v$, we obtain that any $\mu$-eigenfunction $f$ is of the form
$$f(x_1,x_2)=\widetilde{B}(v_1x_1+v_2x_2)\,,$$
where $\widetilde{B}$ is a primitive of $B$. Being $f$ an eigenfunction of the Laplacian, a straightforward computation shows that it satisfies the equation
$$\frac{d^2\widetilde{B}(s)}{ds^2}+\frac{\mu}{|v|^2}\widetilde{B}(s)=0\,,$$
with $s=v\cdot x$. Hence $\widetilde{B}(s)=A_1\cos (\frac{\mu^{\frac{1}{2}}}{|v|}s)+A_2 \sin (\frac{\mu^{\frac{1}{2}}}{|v|}s)$, where $A_1,A_2\in \R$, thus implying that any $\mu$-eigenfunction $f$ is of the form:

$$f(x_1,x_2)=A_1\cos \Big(\frac{\mu^{\frac{1}{2}}}{|v|} v\cdot x\Big)+A_2 \sin \Big(\frac{\mu^{\frac{1}{2}}}{|v|} v \cdot x\Big).$$

In particular, this implies that the multiplicity of the eigenvalue $\mu$ is $2$, which is a contradiction with the fact that it is at least $4$. Hence $DF(x_*,a_*)$ is of rank $2$ at any point $(x_*,a_*)\in F^{-1}(0)$, and therefore the zero-set of $F$ is regular. The lemma then follows.
\end{proof}

\subsection{Existence of escape orbits for generic $b$-Beltrami vector fields}\label{SS.escape}

Let us consider an asymptotically symmetric $b$-Beltrami vector field
\begin{equation*}
X=X_x\partial_x+X_y\partial_y+zX_z\partial_z
\end{equation*}
on an asymptotically flat $b$-manifold, and its dual $b$-form
\[
\alpha=X_xdx+X_ydy+\frac{X_z}{z}dz\,.
\]

The following lemma is straightforward:

\begin{lemma}\label{L:sym}
In the neighborhood $\mathcal N(Z)$ the functions $X_x$, $X_y$ and $X_z$ do not depend on $z$, i.e., $X_x\equiv X_x(x,y), X_y\equiv X_y(x,y), X_z\equiv X_z(x,y)$. Moreover, $X_z$ (which is minus the exceptional Hamiltonian) is an eigenfunction of the flat Laplacian on $\mathbb T^2$ with eigenvalue $\lambda^2$, and the components $X_x$ and $X_y$ are given by
\begin{align*}
X_x(x,y)=\frac{1}{\lambda}{\partial_y X_z(x,y)}\,, \qquad X_y(x,y)=-\frac{1}{\lambda}{\partial_x X_z(x,y)}\,.
\end{align*}
\end{lemma}
\begin{proof}
It is easy to check that the assumption $[X,z\partial_z]=0$ implies that $X_x\equiv X_x(x,y), X_y\equiv X_y(x,y), X_z\equiv X_z(x,y)$. Now, the lemma follows arguing exactly as in the proof of Proposition~\ref{prop:ReebHamdim3} using the $b$-Beltrami equation.
\end{proof}

Applying Lemma~\ref{L:genflat}, for any $\delta'>0$ and $k\geq 2$, we can take a $\lambda^2$-eigenfunction $\widehat{X_z}(x,y)$ of the Laplacian on $\mathbb T^2$ which is Morse and $\delta'$-close to $X_z(x,y)$, i.e.
$$
||\widehat{X_z}-X_z||_{C^{k+1}(\mathbb T^2)}<\delta'\,.
$$
Moreover, we can also assume that the zero set of $\widehat{X_z}(x,y)$ is regular, because it is a property that holds for an open and dense set in $\mathcal{E}_{\lambda^2}$ in the $C^k$-topology, $k\geq2$, cf.~\cite[Proposition 4]{peraltaradu}.

Then, the vector field
\begin{equation}\label{eq:model}
\widehat X:=\frac{1}{\lambda}{\partial_y \widehat{X_z}}\partial_x-\frac{1}{\lambda}{\partial_x \widehat{X_z}}\partial_y+z\widehat X_z\partial_z
\end{equation}
is a $b$-Beltrami vector field for the $b$-flat metric in a neighborhood $\mathcal N(Z)$ of $Z$, and we have the obvious estimate
\begin{equation}\label{est1}
||\widehat X-X||_{C^k(N(Z))}<C\delta'\,.
\end{equation}

Let us take a second neighborhood $\mathcal N_1(Z)\subset \mathcal N(Z)$, and notice that the $1$-form
\[
\widehat\beta:=\frac{1}{\lambda}{\partial_y \widehat{X_z}}dx-\frac{1}{\lambda}{\partial_x \widehat{X_z}}dy+\frac{\widehat X_z}{z}dz\,,
\]
which is the dual of $\widehat X$ using the flat $b$-metric, is contact in $\mathcal N(Z)\backslash \mathcal N_1(Z)$ and is close to $\alpha$ as
\[
||\widehat\beta-\alpha||_{C^k(\mathcal N(Z)\backslash \mathcal N_1(Z))}<C\delta'\,.
\]

We claim that the $1$-form $\widehat\beta$ can be extended as a $b$-contact form $\beta$ on the whole manifold $M$, which is close to $\alpha$ in the following sense:
\begin{equation}\label{est2}
||\beta-\alpha||_{C^k(M\backslash \mathcal N_1(Z))}<C\delta'\,.
\end{equation}
Indeed, the following lemma is standard, we provide a short proof for the sake of completeness:
\begin{lemma}
Let $\alpha$ be a contact form on a manifold $M_0$, and $\widehat \beta$ a contact form on an open set $\mathcal U\subset M_0$ so that $||\alpha-\widehat \beta||_{C^k(\mathcal U)}<\varepsilon_0$. If $\varepsilon_0$ is small enough, and $\mathcal U_1\subset \mathcal U$ is any open set properly contained in $\mathcal U$, there exists a contact form $\beta$ on $M_0$ which satisfies that $\beta|_{\mathcal U_1}=\widehat\beta$ and $||\beta-\alpha||_{C^k(M_0)}<C\varepsilon_0$.
\end{lemma}
\begin{proof}
Take a smooth cutoff function $F:M\to\mathbb [0,1]$ which is equal to $1$ in $\mathcal U_1$, $F>0$ in $\mathcal U$, and $0$ in $M\backslash \mathcal U$. Let us define the $1$-form on $M_0$
\[
\beta:=F\widehat \beta +(1-F)\alpha\,.
\]
Obviously $\beta|_{\mathcal U_1}=\widehat \beta$, $\beta|_{M_0\backslash \mathcal U}=\alpha$ and
\[
||\beta-\alpha||_{C^k(M_0)}=||F(\widehat \beta -\alpha)||_{C^k(\mathcal U\backslash \mathcal U_1)}<C\varepsilon_0\,.
\]
It remains to check that $\beta$ is a contact form on $M_0$, so let us compute
\begin{align*}
\beta \wedge d\beta = F^2\widehat \beta \wedge d\widehat \beta + (1-F)^2\alpha\wedge d\alpha + F(1-F)(\widehat \beta \wedge d\alpha + \alpha\wedge d\widehat \beta)+dF\wedge \widehat \beta \wedge\alpha\,.
\end{align*}
Noticing that the assumption $||\alpha-\widehat \beta||_{C^k(\mathcal U)}<\varepsilon_0$ implies the estimates
\begin{align*}
&||\widehat \beta \wedge \alpha||_{C^0(\mathcal U)}<C\varepsilon_0\\
&||\widehat \beta \wedge d\alpha + \alpha\wedge d\widehat \beta-2\widehat\beta\wedge d\widehat\beta||_{C^0(\mathcal U)}<C\varepsilon_0\,,
\end{align*}
we conclude
\begin{align*}
|\beta\wedge d\beta|&>|F^2\widehat \beta \wedge d\widehat \beta + (1-F)^2\alpha\wedge d\alpha + 2F(1-F)\widehat \beta \wedge d\widehat\beta|-C\varepsilon_0\\
&=|F(2-F)\widehat \beta \wedge d\widehat \beta+(1-F)^2\alpha\wedge d\alpha|-C\varepsilon_0>c_0>0\,,
\end{align*}
where in the last inequality we have used that $F\in [0,1]$ and $\alpha$ and $\widetilde \beta$ are contact forms. This proves that $\beta$ is a contact form on $M_0$ and the lemma follows.
\end{proof}
The existence of the $b$-contact form $\beta$ satisfying the estimate~\eqref{est2} then immediately follows applying this lemma with $M_0=M\backslash \mathcal N_1(Z)$, $\mathcal U=\mathcal N(Z)\backslash \mathcal N_1(Z)$, $\mathcal U_1=\mathcal N_2(Z)\backslash \mathcal N_1(Z)$ and $\varepsilon_0=C\delta'$, where $\mathcal N_2(Z)$ is a neighborhood of $Z$ properly contained in $\mathcal N(Z)$ and containing $\mathcal N_1(Z)$.

The Beltrami-contact correspondence Theorem~\ref{thm:bBeltramicorrespondence} then implies that the Reeb field $Y$ of the $b$-contact form $\beta$ is $b$-Beltrami for some (weakly) compatible metric $\widehat g$. Since $\beta|_{\mathcal N_1(Z)}=\widehat\beta$, whose dual vector field with the flat $b$-metric is the $b$-Beltrami vector field $\widehat X$ by construction, then $\widehat g$ can be taken to be $b$-flat in the neighborhood $\mathcal N_1(Z)$ and $Y|_{\mathcal N_1(Z)}=\widehat X$ . Therefore, $\widehat g$ is asymptotically $b$-flat and $C\delta'$-close to $g$, although generally different from $g$ in $M\backslash \mathcal N_1(Z)$.

Using the estimates~\eqref{est1} and~\eqref{est2}, we then conclude that for any $\delta>0$ small enough, there is a $b$-Beltrami vector field $Y$ on $(M,Z)$ endowed with an asymptotically flat $b$-metric such that
$$||Y-X||_{C^k(M)}<C\delta'=:\delta$$
and its exceptional Hamiltonian is a Morse function on $Z$ with regular zero set. This proves the $C^k$-density of the aforementioned properties of the exceptional Hamiltonian. The openness is immediate because being Morse and having regular zero set are open conditions in the $C^k$ topology, $k\geq 2$.
\begin{remark}
In what follows, by a generic asymptotically symmetric $b$-Beltrami vector field on an asymptotically flat $b$-manifold we will mean a vector field whose exceptional Hamiltonian is Morse and has regular zero set.
\end{remark}

Once the genericity of the aforementioned class of $b$-Beltrami vector fields on asymptotically flat $b$-manifolds has been established, the proof of the second part of Theorem~\ref{thm:mainthm2} is exactly the same as the proof of Theorem~\ref{thm:mainthm} via the local analysis of the zero points of the generic $b$-Beltrami vector field on $Z$. One just has to note that the first Betti number of $\mathbb T^2$ is $2$, and hence we get the lower bound of $4$ escape orbits.

\subsection{Generic existence of generalized singular periodic orbits}

In this section we prove the first part of Theorem~\ref{thm:mainthm2}. We follow the same notation as in Section~\ref{SS.escape} without further mention.

Let us consider a generic asymptotically symmetric $b$-Beltrami vector field $X$ on an asymptotically flat $b$-manifold $(M,Z,g)$. Then the component $X_z$ is a Morse eigenfunction of the flat Laplacian on $\mathbb T^2$. Take a point $p_0=(x_0,y_0,0)\in Z$ where $X_z(x,y)$ attains its minimum value, which is negative because any (nonconstant) eigenfunction has zero mean:
\[
X_z(x_0,y_0)<0\,.
\]
Our goal is to analyze the integral curves of $X$ near the point $p_0$. To this end we recall that in a neighborhood $\mathcal N(Z)$ of the critical surface an asymptotically symmetric $b$-Beltrami vector field has the form, cf. Lemma~\ref{L:sym},
\[
\frac{1}{\lambda}{\partial_y X_z(x,y)}\partial_x-\frac{1}{\lambda}{\partial_x X_z(x,y)}\partial_y+zX_z(x,y)\partial_z\,.
\]

It is obvious that $X_z$ is a first integral of $X$. Moreover, doing a Taylor expansion of $X_z$ at $p_0$, we obtain that
\[
\dot z=zX_z(x,y)=zX_z(x_0,y_0)+O(|z(x-x_0)|+|z(y-y_0)|)\leq 0
\]
if $(x,y,z)\in B\cap\{z\geq 0\}$, and in fact $\dot z=0$ only on the disk $D:=B\cap \{z=0\}$. Here $B$ is a small enough neighborhood of $p_0$. Therefore, the value of the defining coordinate $z$ decreases along the integral curves of $X$ near $p_0$. Noticing that the level sets of $X_z$ near $p_0$ are cylindrical, it is easy to infer that there is a positive constant $\delta>0$ such that the compact set
\[
K:=\{(x,y,z)\in \tilde B: \delta/2 \leq z\leq \delta \}\,,
\]
where $\tilde B\subset B$ is a smaller closed neighborhood of $p_0$, is attracted by the critical surface $Z$, i.e., the $\omega$-limit of $K$ is contained in the disk $D$.

Since $X$ is volume-preserving in $M\backslash Z$ (because it is divergence-free with respect to the $b$-volume form associated to the $b$-metric $g$), it follows that the $\alpha$-limit of the compact set $K$ must intersect the critical surface $Z$. Specifically, there exists a sequence of points $\{p_{k}\}_{k=1}^\infty\subset K$ and a decreasing sequence of negative times $\{t_{k}\}_{k=1}^\infty$ with $\lim_{k} t_k =-\infty$ such that $\varphi_{t_k}(p_k)\to p_*\in Z$. Here, $\varphi_t$ denotes the flow defined by the vector field $X$ on $M$. By compactness, after taking a subsequence if necessary, we have that $p_{k}\to \widetilde p\in K$, and therefore the continuity of $\varphi_t$ with respect to initial conditions implies that (up to a subsequence),
\[
\varphi_{t_k}(\widetilde p) \to p_*\in Z\,.
\]
Since the $\omega$-limit of $\widetilde p$ is contained in $D\subset Z$, we conclude that the orbit $\{\varphi_t(p):-\infty <t<\infty\}$ is a generalized singular periodic orbit, so Theorem~\ref{thm:mainthm2} follows.

We finish this section noticing that, using the same techniques, we can prove the singular Weinstein conjecture on globally flat $b$-manifolds, for globally symmetric $b$-Beltrami vector fields. By globally symmetric we mean that $X$ commutes with the globally defined vector field $\sin z\partial_z$ on the whole $\mathbb T^3$.

\begin{proposition}\label{prop:genericsingularweinstein}
Let $X$ be a globally symmetric, $b$-Beltrami vector field on a globally flat $b$-manifold that is not identically zero on $Z$. Then $X$ has at least two singular periodic orbits.
\end{proposition}

\begin{proof}
Proceeding as in Lemma~\ref{L:sym} but using the globally defined defining function $\sin z$ instead of $z$, it follows that the global symmetry of $X$ implies that it has the form
$$X=\frac{1}{\lambda}\partial_y H \partial_x-\frac{1}{\lambda}\partial_x H \partial_y+H\sin z\partial_z$$
on $\mathbb{T}^3$, where $H=H(x,y)$ satisfies $\Delta H+\lambda^2 H=0$. Then one can do the same computation as in Example~\ref{ex:T3flat}. More precisely, $H(x,y)$ is a first integral of $X$ and each integral curve with initial condition $(p_\pm,z_0)$, where $p_\pm$ is a point on $\mathbb T^2$ where $H$ attains its global minimum (resp. maximum) and $z_0\in (0,\pi)$, is a singular periodic orbit. Here we are using that, being a nontrivial eigenfunction of the Laplacian, the minimum value of $H$ is negative (resp. the maximum value is positive).
\end{proof}
\begin{remark}
We stress that in Proposition~\ref{prop:genericsingularweinstein} we do not need to assume that the $b$-Beltrami vector field is generic. In fact, Theorem~\ref{thm:mainthm2} holds for any asymptotically symmetric $b$-Beltrami vector field provided that it is not identically zero on $Z$. In this case the number of escape orbits that one obtains is at least $2$ instead of $4$.
\end{remark}

\section{Concluding remarks and open problems}\label{sec:concludingrmks}

Singularities naturally occur in regularization procedures in celestial mechanics, see \cite{knauf}. In particular, the McGehee blow-up \cite{McGehee} is classically used in the $N$-body problem to study the manifold at infinity. It is a non-canonical symplectic change of coordinates and therefore the induced geometric structures  are no longer symplectic but of $b^3$-type, where the critical hypersurface is being identified with the manifold at infinity in the original problem. In \cite{MO2}, this is used in combination with Proposition \ref{prop:ReebHamdim3} to prove the existence of infinitely many periodic orbits on positive energy level sets in the restricted planar circular three body problem at the manifold at infinity. Orbits coming and going to the infinity manifold are singular periodic orbits as studied in the present article, and thereby emphasizes the importance of the singular Weinstein conjecture in view of applications in celestial mechanics. This description is reminiscent of \emph{hyperbolic scattering orbits} in the $N$-body problem, as studied in \cite{dmmg} and \cite{MV}. Those are orbits in the $N$-body problem where the mutual distances between the bodies go to infinity. In the language of the present note, those orbits are just singular periodic orbits.

Our article also provides the general framework which includes oscillatory motions in celestial mechanics as they constitute  a particular case of generalized singular periodic orbits. Those are orbits $(q(t),p(t))$ in the phase space $T^*\R^n$ such that
$$\limsup_{t\to\pm \infty} \norm{q(t)}=\infty \text{ and } \liminf_{t\to\pm \infty} \norm{q(t)}<\infty.$$
Under the classical McGehee change of coordinates $r=\frac{2}{x^2}$, where $r$ is the radius of the position in polar coordinates, the condition on the upper limit  is equivalent to saying that the orbit is a generalized singular periodic orbit.
The existence of oscillatory motions in Celestial Mechanics has a long story: In 1960 Sitnikov \cite{s} proved existence of oscillatory motions for the restricted spatial elliptic three body problem when primaries have mass $\mu= 1/2$ and move on ellipses of small enough eccentricity while the third body moves on the (invariant) vertical axis. In 1973 Moser \cite{moser} provided a new proof of Sitnikov results. First results in the planar case were obtained by Llibre and Simó \cite{llibre} in 1980 following Moser’s approach. It was not until the recent work \cite{GMS}
 that their existence has been proved in the restricted planar circular three-body problem for any mass ratio provided the Jacobi constant $J$ is sufficiently large (see also
 \cite{GMSS} for the restricted planar elliptic three-body case).

The authors believe that the techniques and results of the present paper will be useful to tackle interesting problems in classical mechanics, as the existence of scattering orbits, {escape orbits and oscillatory motions in a more general set-up (see for instance \cite{chenciner3})}. We plan to deal with this in future work.

\end{document}